\documentclass[pdflatex,sn-mathphys,Numbered]{sn-jnl}% Math and Physical Sciences Reference Style
\usepackage{graphicx}%
\usepackage{multirow}%
\usepackage{amsmath,amssymb,amsfonts}%
\usepackage{amsthm}%
\usepackage{mathrsfs}%
\usepackage[title]{appendix}%
\usepackage{xcolor}%
\usepackage{textcomp}%
\usepackage{manyfoot}%
\usepackage{booktabs}%
\usepackage{algorithm}%
\usepackage{algorithmicx}%
\usepackage{algpseudocode}%
\usepackage{listings}%
\theoremstyle{thmstyleone}%
\newtheorem{theorem}{Theorem}%  meant for continuous numbers
\theoremstyle{thmstyletwo}%
\newtheorem{corollary}{Corollary}%

\theoremstyle{thmstylethree}%

\newcommand\blfootnote[1]{%
  \begingroup
  \renewcommand\thefootnote{}\footnote{#1}%
  \addtocounter{footnote}{-1}%
  \endgroup
}

\begin{document}

\title{On the accurate computation of the Newton form  of the Lagrange interpolant}


%\author*[1,2]{\fnm{First} \sur{Author}}\email{iauthor@gmail.com}
%\affil*[1]{\orgdiv{Department}, \orgname{Organization}, \orgaddress{\street{Street}, \city{City}, \postcode{100190}, \state{State}, \country{Country}}}
%
%\affil[2]{\orgdiv{Department}, \orgname{Organization}, \orgaddress{\street{Street}, \city{City}, \postcode{10587}, \state{State}, \country{Country}}}
%
%\affil[3]{\orgdiv{Department}, \orgname{Organization}, \orgaddress{\street{Street}, \city{City}, \postcode{610101}, \state{State}, \country{Country}}}

\author[1]{\fnm{Y.} \sur{Khiar}}\email{ykhiar@unizar.es}

\author[1]{\fnm{E.} \sur{Mainar}}\email{esmemain@unizar.es}

\author*[2]{\fnm{E.} \sur{Royo-Amondarain}}\email{eduroyo@unizar.es}

\author[1]{\fnm{B.} \sur{Rubio}}\email{brubio@unizar.es}

\affil[1]{\orgdiv{Departamento de Matem\'{a}tica Aplicada/IUMA}, \orgname{Universidad de Zaragoza}, \orgaddress{\city{Zaragoza}, \country{Spain}}}

\affil[2]{\orgdiv{Departamento de Matem\'{a}ticas}, \orgname{Universidad de Zaragoza}, \orgaddress{\city{Zaragoza}, \country{Spain}}}


\abstract{In recent years many efforts have been devoted to finding bidiagonal factorizations of nonsingular totally positive matrices, since their accurate computation allows to numerically solve several important algebraic problems with great precision, even for large ill-conditioned matrices. In this framework, the present work provides the factorization of the collocation matrices of Newton bases---of relevance when considering the Lagrange interpolation problem---together with an algorithm that allows to numerically compute it to high relative accuracy. This further allows to determine the coefficients of the interpolating polynomial and to compute the singular values and the inverse of the collocation matrix. Conditions that guarantee high relative accuracy for these methods and, in the former case, for the classical recursion formula of divided differences, are determined.
Numerical errors due to imprecise computer arithmetic or perturbed input data in the computation of the factorization are analyzed.
Finally, numerical experiments illustrate the accuracy and effectiveness of the proposed methods with several algebraic problems, in stark contrast with traditional approaches.
\blfootnote{\textit{Preprint submitted to Numerical Algorithms on November 22, 2023}}}




\maketitle

\section{Introduction}\label{sec1}

A classical approach to the Lagrange interpolation problem is the Newton form, in which the polynomial interpolant is written in terms of the Newton basis. Its coefficients, called divided differences, can be determined both by means of a recursive formula or  by solving a linear system that involves the collocation matrix of the basis.
The recursive computation of divided differences requires many subtractions which can lead to cancellation: an error-inducing phenomenon that takes place when two nearly equal numbers are subtracted and that usually elevates the effect of earlier errors in the computations.
In fact, this is also the case of the notoriously ill-conditioned collocation matrices that are associated with the Lagrange interpolation problem---these errors can heavily escalate when the number of nodes increases, eventually turning unfeasible the attainment of a solution in any algebraic problem involving these matrices.
It is worth noting that sometimes the errors in computing the divided differences can be ameliorated by considering different node orderings \cite{Fischer1989,Tal-Ezer1991}, although this strategy may not be sufficient to obtain the required level of precision in a given high-order interpolation problem.

However, in some scenarios it is possible to keep numerical errors under control. In a given floating-point arithmetic, a real value is said to be determined to high relative accuracy (HRA) whenever the relative error of the computed value is bounded by the product of the unit round-off and a positive constant, independent of the arithmetic precision. HRA implies great accuracy in the computations since the relative errors have the same order as the machine precision and the accuracy is not affected by the dimension or the conditioning of the problem to be solved. A sufficient condition to assure that an algorithm can be computed to HRA is the non inaccurate cancellation condition. Sometimes denoted as NIC condition, it is satisfied if  the algorithm does not require  inaccurate subtractions and only evaluates products, quotients, sums of numbers of the same sign, subtractions of numbers of opposite sign or subtraction of initial data (cf.~\cite{Koev1}, \cite{Koev4}). In this sense, in this paper we provide the precise conditions under which the recursive computation of divided differences can be performed to HRA---these include strictly ordered nodes and a particular sign structure of the node values of the interpolated function.

In the research of algorithms that preserve HRA, a major step forward was given in the seminal work of Gasca and Pe\~na for nonsingular totally positive matrices, since these can be written as a product of bidiagonal matrices \cite{Gasca3}. This factorization can be seen as a representation of the matrices exploiting their total positivity property to achieve accurate numerical linear algebra since, if it is provided to HRA, it allows solving several algebraic problems involving the collocation matrix of a given basis to HRA. In the last years, the search for bidiagonal decompositions of different totally positive bases has been a very active field of research \cite{dop1,dop2,Koev1,Demmel2,MPRJSC,MPRAdvances,MPRRACSAM}.

The present work can be framed as a contribution to the above picture for the particular case of the well-known Newton basis and its collocation matrices. The conditions guaranteeing their total positivity are derived and a fast algorithm to obtain their  bidiagonal factorization to HRA is provided. Under these conditions, the obtained bidiagonal decomposition can be applied to compute to HRA the singular values, the inverse, and also the solution of some linear systems of equations. As will be illustrated,   the accurate resolution of these systems provides an alternative procedure to calculate divided differences to high relative accuracy.
 
In order to make this paper as self-contained as possible, Section \ref{Section:notations} recalls basic concepts and  results related to total positivity,  high relative accuracy and interpolation formulae.  Section \ref{sec:NewtonsForm} focuses on the recursive computation of divided differences  providing conditions  for their computation to high relative accuracy.  In Section \ref{sec:BDANewton},  the bidiagonal factorization of collocation matrices of Newton bases is  derived and the analysis of their total positivity is performed.   A fast algorithm for the computation of the bidiagonal factorization is provided in Section \ref{sec:error}, where the numerical errors appearing in  a floating-point arithmetic   are also studied and   a structured condition number for  the considered matrices is deduced. Finally, Section \ref{sec:experimentos} illustrates the numerical  performed experimentation.

\section{Notations and auxiliary results} \label{Section:notations}

 Let us recall that a matrix is  totally positive (respectively, strictly totally positive) if all its minors are nonnegative (respectively, positive). 
By Theorem 4.2 and the arguments of p.116  of \cite{Gasca3},   a nonsingular    totally positive  $A\in {\mathbb R}^{(n+1)\times(n+1)}$
 can be written as follows,
 \begin{equation}\label{BDAfac}
A=F_nF_{n-1}\cdots F_1 D G_1 G_2\cdots   G_n,
\end{equation}
 where $F_i\in {\mathbb R}^{(n+1)\times(n+1)}$ and $G_i\in {\mathbb R}^{(n+1)\times(n+1)}$, $i=1,\ldots,n$,  are the totally positive,  lower and upper triangular bidiagonal matrices described by
\begin{equation}  
\label{fgi}
{
   F_i=\left(\begin{array}{cccccccccc}
1  \\
   &   & \ddots   \\
   &   &        & 1 &      \\
       &        &    &  m_{i+1,1}         &  &  1   \\
    &   &        &    &           &        \ddots        & \ddots &       \\
&     &        &    &           &         &   m_{n+1,n+1-i}  &  1
    \end{array} \right),  \,
G_i^T=\left(\begin{array}{cccccccccc}
1  \\
   &   & \ddots   \\
   &   &        & 1 &      \\
       &        &    &  \widetilde{m}_{i+1,1}         &  &  1   \\
    &   &        &    &           &        \ddots        & \ddots &       \\
&     &        &    &           &         &   \widetilde{m}_{n+1,n+1-i}  &  1
    \end{array} \right)}
\end{equation}  
and $D\in {\mathbb R}^{(n+1)\times(n+1)}$ is a diagonal matrix with positive diagonal entries. 

In \cite{Koev4}, the bidiagonal factorization \eqref{BDAfac} of a  nonsingular and   totally positive $A\in {\mathbb R}^{(n+1)\times(n+1)}$ is represented by defining a matrix $BD(A)=(BD(A)_{i,j})_{1\le i,j\le n+1}$ such that 
\begin{equation}\label{eq:BDA}
 BD(A)_{i,j}:=\begin{cases} m_{i,j}, & \textrm{if } i>j, \\[5pt]  p_{i,i}, & \textrm{if } i=j, \\[5pt]  \widetilde{m}_{ j ,i}, & \textrm{if } i<j.    \end{cases}
\end{equation}
This representation  will allow us to define algorithms adapted to the totally positive structure, providing accurate computations with  $A$. 
 
If  the bidiagonal factorization \eqref{BDAfac} of a nonsingular and  totally positive matrix $A$ can be computed to high relative accuracy, the computation of   its eigenvalues and singular values, the computation of   $A^{-1}$ and even the resolution of systems of linear equations  $Ax=b$, for vectors $b$ with alternating signs, can be also computed to high relative accuracy using the algorithms provided in \cite{Koev2}.

Let  $(u_0, \ldots, u_n)$ be  a basis  of a space $U$ of  functions defined on $I\subseteq R$.
Given  a sequence of parameters
  $ t_1<\cdots<t_{n+1} $  on $I$,  the corresponding  collocation matrix   is defined by
\begin{equation} \label{eq:MC}
M\left( t_1,\ldots , t_{n+1} \right)  :=\big(  u_{j-1}(t_i) \big)_{1\le i,j\le n+1}.
\end{equation}

 Let   ${\mathbf P}^n(I)$ be  the  $(n+1)$-dimensional space formed by all polynomials of degree not greater than $n$,  in a variable defined on  $I\subseteq\mathbb R$, that is,
 $$ {\mathbf P}^n(I):=\textrm{span}\{1, t, \ldots, t^n\},\quad   t\in I.
 $$
 Given nodes $t_1, \ldots ,t_{n+1}$ on $I$ and a function $f:I\to\mathbb R$, we are going to address the Lagrange interpolation problem for finding   $p_n\in {\mathbf P}^n(\mathbb R)$  such that 
 $$
 p_n(t_i)=f(t_i),\quad i=1,\ldots,n+1.
 $$
When considering the monomial basis  $(m_{0} ,\ldots,m_{n})$, with $m_i(t)=t^i$, for $i=0,\ldots,n$,  the interpolant  can be written as follows,
 \begin{equation} \label{eq:MonomialBasis}
 p_n(t)=\sum_{i=0}^n c_{i+1} m_i(t),
 \end{equation}
and the coefficients $c_i$, $i=1,\ldots,n+1$, form the solution vector $c=(c_1,\ldots,c_{n+1})^T$ of the linear system
$$
Vc=f,
$$
where $f:=(f(t_1),\ldots,f(t_{n+1}))^T$ and $V\in \mathbb R^{(n+1)\times(n+1)} $ is the collocation matrix  of the monomial basis at the nodes $t_i$, $i=1,\ldots,n+1$, that is,   
\begin{equation}\label{eq:V}
V :=\left( t_{i}^{j-1} \right)_{1\le i,j \le n+1}.
\end{equation}
Let us observe that $V$ is the Vandermonde matrix at the considered nodes and recall that Vandermonde matrices   have relevant  applications in linear interpolation
and numerical quadrature (see for example \cite{Finck} and \cite{Oruc}). In fact, 
for any increasing sequence of positive values,  $0 <t_1< \cdots<t_{n+1}$, the corresponding Vandermonde matrix $V$  in \eqref{eq:V}
  is known to be strictly totally positive (see Section 3 of  \cite{Koev4}) and  so,
 we can write
 \begin{equation}\label{eq:c}
{c}= V^{-1}f.
 \end{equation}
 
In \cite{Koev4} or Theorem 3 of \cite{MPRNLA} it is shown that the Vandermonde matrix $V$ admits a bidiagonal factorization of the form \eqref{BDAfac}:
  \begin{equation}\label{eq:BDAVandermonde}
  V =  F_nF_{n-1}\cdots F_1 D G_1 G_2\cdots   G_n,
 \end{equation}
 where  $F_i\in {\mathbb R}^{(n+1)\times(n+1)}$ and $G_i\in {\mathbb R}^{(n+1)\times(n+1)}$, $i=1,\ldots,n$,  are the  lower and upper triangular bidiagonal matrices described by 
\eqref{fgi} with 
\begin{equation}\label{eq:mijVan}
m_{i,j}=  \prod_{k=1}^{j-1}  \frac{t_{i } - t_{i-k }  }{t_{i-1} -t_{i-k-1}  },         \quad  \tilde m_{i,j}=t_j,\quad 1\le j< i \le n+1,
\end{equation}
and $D$ is the diagonal matrix whose entries are 
\begin{equation}\label{eq:diiVan}
d_{i,i}=   \prod_{k=1}^{i-1} (t_i - t_k),\quad  i=1,\ldots,n+1.
\end{equation}
Using Koev's notation,  this factorization of $V$ can be represented  through the matrix $BD(V)\in\mathbb R^{(n+1)\times (n+1)}$ with 
 \begin{equation}\label{eq:BDV}
 BD(V )_{i,j}:=\begin{cases}
    t_{i }  , & \textrm{if } i<j,   \\[7pt] 
  \prod_{k=1}^{i-1} (t_i - t_k), & \textrm{if } i=j,  \\[7pt]  
 \prod_{k=1}^{j-1}  \frac{t_{i } - t_{i-k }  }{t_{i-1} -t_{i-k-1}  }, & \textrm{if } i>j.
 \end{cases}
\end{equation}
 
Moreover, it can be easily checked that the computation of  $BD(V )$ does not require inaccurate cancellations and can be performed to high relative accuracy.  
 
When considering interpolation nodes $t_1,\ldots, t_{n+1}$ such that $t_i\ne t_j$  for $i\ne j$, and the corresponding  Lagrange polynomial basis  $(\ell_0,\ldots, \ell_n)$, with 
\begin{equation}\label{eq:LagBasis}
\ell_i(t):=\prod_{j\ne i} \frac{t-t_{j+1}}{t_{i+1}-t_{j+1}},\quad i=0,\ldots,n,
\end{equation}
the Lagrange formula of the polynomial interpolant $p_n$  is 
\begin{equation}\label{eq:Lag}
p_n(t)=\sum_{i=0}^n f(t_{i+1}) \ell_i(t).    
\end{equation}
Taking into account   \eqref{eq:Lag},   \eqref{eq:MonomialBasis}  and \eqref{eq:c}, we have
$$
p_n(t)= (\ell_0(t),\ldots, \ell_n(t)) { f}= (m_0(t),\ldots, m_n(t)) V^{-1} { f}, 
$$
with $f=(f(t_1),\ldots,f(t_{n+1}))^T$, and  deduce that 
\begin{equation}\label{eq:m=lV}
(m_0,\ldots, m_n) = (\ell_0,\ldots, \ell_n) V.
\end{equation}
Let us note that identity \eqref{eq:m=lV}  means that the Vandermonde matrix $V\in\mathbb R^{(n+1)\times (n+1)}$ is the change of basis matrix between the $(n+1)$-dimensional monomial and  Lagrange basis of the polynomial space ${\mathbf P}^n(\mathbb R)$ corresponding to the considered interpolation nodes.

 The Lagrange's interpolation formula \eqref{eq:Lag} is usually considered  for small numbers of  interpolation nodes,  due to  certain shortcomings claimed such as:
\begin{itemize}
\item The evaluation of the interpolant $p_n(t)$ requires  $O(n^2)$ flops.
\item   A new computation from scratch is needed  when adding a new interpolation data pair.
\item The computations are numerically unstable.
\end{itemize}
Nevertheless, as explained in \cite{Berrut}, the Lagrange formula can be improved taking into account the following identities, 
\begin{equation}\label{eq:LagImp}
p_n(t)=\sum_{i=0}^n f(t_{i+1}) \ell_i(t)=\ell(t) \sum_{i=0}^n f(t_{i+1}) \frac{\omega_{i+1}}{t-t_{i+1} },    
\end{equation}
where 
$$
\ell(t) :=\prod_{i=0}^n (t-t_{i+1}), \quad \omega_i: =1/\prod_{k\ne i} (t_{i}-t_{k})=1/\ell'(t_i), \quad i=1,\ldots,n+1.
$$
Furthermore,  since $\sum_{i=0}^n\ell_i(t)=1$, for all $t$, the following barycentric formula can be derived for the Lagrange interpolant
\begin{equation}\label{eq:LagBaryc} 
p_n(t)=\frac{ 
 \sum_{i=0}^n 
  f(t_{i+1})  
  \frac{ \omega_{i+1} }{ t-t_{i+1} } } {  \sum_{i=0}^n  \frac{\omega_{i+1}}{t-t_{i+1} }  }.  
\end{equation}
Let us observe that  any common factor in the values $\omega_j$, $i=1,\ldots,n+1$, can  be cancelled without changing the value of the interpolant $p_n(t)$. This  property is used in \cite{Berrut} to derive interesting properties of the barycentric formula \eqref{eq:LagBaryc} for the interpolant.

\section{Accurate computation of  divided differences} \label{sec:NewtonsForm}

Instead of Lagrange's formulae \eqref{eq:LagImp} and \eqref{eq:LagBaryc}, one can use 
 the Newton  form  of the interpolant, which is obtained when the interpolant is written in terms of the Newton basis $(w_0,\ldots,w_n)$ determined by the interpolation  nodes $t_1,\ldots,t_{n+1}$, 
\begin{equation} \label{eq:NewtonBasis}
w_0(t):=1,\quad  w_i(t):= \prod_{k=1}^i t-t_k ,\quad i=1,\ldots,n,
\end{equation}
 as follows,  
 \begin{equation} \label{eq:pol}
p_n(t)=\sum_{i=0}^{n}[t_1,\ldots,t_{i+1} ]f \,w_i(t),
\end{equation}
where $[t_1,\ldots,t_i ]f$ denotes the divided difference   of the interpolated function $f$ at the nodes $t_1,\ldots,t_i$.
If $f$ is $n$-times continuously differentiable on $[t_1,t_{n+1}]$, the divided differences  $[t_1,\ldots,t_i ]f $, $i=1,\ldots,n+1$, can be obtained using the following recursion
\begin{equation}\label{eq:divideddif}
 [t_i, \ldots, t_{i+k}]f =\begin{cases} 
      \frac{   [t_{i+1},  \ldots, t_{i+k}]f- [t_i,\ldots, t_{i+k-1}]f}{t_{i+k}-t_i},  & \textrm{if }  t_{i+k}\ne t_i,\\[5pt]
          \frac{  f^{(k)}(t_i)}{ k! },& \textrm{if }  t_{i+k}= t_i.\end{cases}
\end{equation}
Note that the divided differences   only depend on   the interpolation nodes and, once  computed, the interpolant   \eqref{eq:pol} can be evaluated in $O(n)$ flops per evaluation. 

For a given sequence $d_i := f(t_i)$,  $i=1,\ldots,n+1$, let us define
\begin{equation}\label{eq:dik}
d_{i,k} := [t_i,\dots,t_{i+k}]f, \quad k=0,\dots, n,\quad  i=1,\dots,n+1-k.
\end{equation}
The following result shows that the computation of  $d_{i,k}$ using the recursion \eqref{eq:divideddif}    satisfies the NIC condition for ordered sequences of  nodes.

\begin{theorem}\label{theo:divdif}
  Let $t_1,\dots , t_{n+1}$ be nodes in strict increasing (respectively, decreasing) order and   $d_1,\dots, d_{n+1}$ given values having alternating signs. Then, the values   
\begin{equation}\label{eq:recdik}
d_{i,0}:=d_i, \; i=1,\dots,n+1, \quad     d_{i,k} = \frac{   d_{i+1,i+k} - d_{i, i+k-1} }{ t_{i+k}-t_i},  \; k=1,\dots, n, \; i=1,\dots,n+1-k,
\end{equation}
can be computed to high relative accuracy.
\end{theorem}
\begin{proof}
We proceed by induction on $k$. Under the considered hypotheses, the result follows for $k=0$.
Now, let us assume that the claim holds for $k$. Then, by \eqref{eq:recdik}, we can write
$$
d_{i,k+1}=\frac{d_{i+1,k}-d_{i,k}}{t_{i+k+1}-t_i},\quad
d_{i+1,k+1}=\frac{d_{i+2,k}-d_{i+1,k}}{t_{i+k+2}-t_{i+1}},
$$
and, by the induction hypothesis,  deduce that the values $d_{i+2,k}$ and $d_{i,k}$ have different sign than   $d_{i+1,k}$. Now, taking into account that the nodes $t_i$ are in strict  order,  we derive that the elements  $d_{i,k+1}$ and $d_{i+1,k+1}$ have alternating signs.
  
On the other hand,  the computation to high relative accuracy  of  the quantity $d_{i,k}$
can be guaranteed since it is computed as a quotient whose numerator is a sum of numbers of the same sign, and
 the subtractions in the denominator involve only initial data and so, will not lead to subtractive cancellations.
\end{proof}

Note that the previous result provides the conditions on a function $f$ for the computation to high relative accuracy of the coefficients of the Newton form  \eqref{eq:NewtonBasis} of its Lagrange polynomial interpolant.  
 \begin{corollary}\label{cor:divdif}
  Let $t_1,\dots , t_{n+1}$ be nodes in strict increasing (respectively, decreasing) order and $f$ a function such that the entries of the vector $ (f(t_1),\dots,f(t_{n+1}))$ can be computed to high relative accuracy and have alternating signs. Then, the divided differences  
 $ [t_i, \ldots, t_{i+k}]f$, $k=0,\dots, n$, $i=1,\dots,n+1-k$,   can be computed to high relative accuracy using \eqref{eq:divideddif}.
 \end{corollary}

Section \ref{sec:experimentos} will illustrate  the accuracy obtained using the recursive computation of divided differences and compare the results with those obtained through an alternative method proposed in the next section.

 \section{Total positivity and bidiagonal factorization of collocation matrices of Newton bases} \label{sec:BDANewton}

Let us observe that, since the polynomial $m_i(t)=t^i$, $i=0,\ldots,n$,    coincides with its   interpolant at $t_1,\ldots, t_{n+1}$, taking into account the Newton  formula  \eqref{eq:pol} for  the monomials $m_i$, $i=0,\ldots,n$, we deduce that 
 \begin{equation}\label{eq:changemN}
 (m_0,\ldots,m_n)=(w_0,\ldots,w_n)U,
  \end{equation}
where the change of basis matrix  $U=(u_{i,j})_{1\le i,j \le n+1}$ is upper triangular and   satisfies   
 $ u_{i,j}= [t_1,\ldots, t_{i}]m_{j-1}$,
 that is,
\begin{equation}\label{eq:U}
U=\left(     \begin{array}{ccccc}  
 1 & [t_1]m_1  &  [t_1]m_2 & \cdots & [t_1]m_n \\   
 0 & 1 & [t_1,t_2]m_2 & \cdots & [t_1,t_2]m_n \\         
 0 & 0 & 1  & \ddots &    \vdots  \\   
 \vdots &  \vdots  &\ddots   & \ddots & [t_1,\ldots,t_{n}]m_n  \\   
 0 &0 &\cdots  & 0 & 1      
        \end{array} \right).
\end{equation}
On the other hand, the  collocation matrix of the Newton basis $(w_0,\ldots,w_n)$  \eqref{eq:NewtonBasis} at the interpolation nodes $t_1,\ldots,t_{n+1}$   is a lower triangular matrix $L=(l_{i,j})_{1\le i,j\le n+1}$ whose entries are
\begin{equation}\label{eq:ColNewton}
  l_{i,j}= w_{j-1}(t_i)= \prod_{k=1}^{j-1}  (t_i-t_k).
\end{equation}
Taking into account \eqref{eq:m=lV} and \eqref{eq:changemN},  we obtain the following Crout factorization of Vandermonde matrices at nodes $t_{i}$, $i=1,\ldots,n+1$, with $t_i\ne t_j$ for $i\ne j$,
\begin{equation}\label{CroutV}
V= M \left[ \begin{array}{c} \ell_0,\ldots, \ell_n \\t_1,\ldots,t_{n+1}\end{array} \right] V= M \left[ \begin{array}{c} m_0,\ldots,m_n \\t_1,\ldots,t_{n+1}\end{array} \right]= M \left[ \begin{array}{c} w_0,\ldots,w_n \\t_1,\ldots,t_{n+1}\end{array} \right]U=LU,
\end{equation}
  where $L$ is the lower triangular collocation matrix of the Newton basis $(w_0,\ldots,w_n)$ and $U$ is the upper triangular change of basis matrix  satisfying \eqref{eq:changemN}.

The following result deduces the bidiagonal factorization of the collocation matrix $L$   of the Newton basis. 

\begin{theorem} \label{thm:BDAL} Given interpolation nodes $t_1,\ldots, t_{n+1}$, with $t_i\ne t_j$ for $i\ne j$,
let $L\in{\mathbb R}^{(n+1)\times(n+1)}$  be the collocation matrix described by \eqref{eq:ColNewton} of the Newton basis \eqref{eq:NewtonBasis}. Then, 
   \begin{equation}\label{eq:BDLNewton}
L=  F_n \cdots F_{1}D, 
\end{equation}
where   $F_i\in{\mathbb R}^{(n+1)\times(n+1)}$, $i=1,\ldots,n$, are lower triangular bidiagonal matrices whose structure is described by \eqref{fgi} and their off-diagonal  entries   are  
\begin{equation}       
   {m}_{i,j} = \prod_{k=1}^{j-1}  \frac{t_{i } - t_{i-k }  }{t_{i-1} -t_{i-k-1}  },       \quad 1\le j< i \le n+1,   \label{eq:mij}  
\end{equation}
and $D\in{\mathbb R}^{(n+1)\times(n+1)}$ is the diagonal matrix $D=\textrm{diag}(d_{1,1},\ldots, d_{n+1,n+1})$ with
\begin{equation}       
 {d}_{i,i} = \prod_{k=1}^{i-1} (t_i - t_k),       \quad 1\le   i \le n+1.   \label{eq:dii}  
\end{equation}
  \end{theorem}
 \begin{proof}
From identities  \eqref{eq:BDAVandermonde},  \eqref{eq:BDV} and  \eqref{CroutV}, we have
 \begin{eqnarray*} 
  V =  M \left[ \begin{array}{c} m_0,\ldots,m_n \\t_1,\ldots,t_{n+1}\end{array} \right]=F_nF_{n-1}\cdots F_1 D G_1 G_2\cdots   G_n, 
 \end{eqnarray*}
 where  $F_i\in {\mathbb R}^{(n+1)\times(n+1)}$ and $G_i\in {\mathbb R}^{(n+1)\times(n+1)}$, $i=1,\ldots,n$,  are the  lower and upper triangular bidiagonal matrices described by 
\eqref{fgi} with 
$$
m_{i,j}=  \prod_{k=1}^{j-1}  \frac{t_{i } - t_{i-k }  }{t_{i-1} -t_{i-k-1}  },         \quad  \tilde m_{i,j}=t_j,\quad 1\le j< i \le n+1,
$$ 
and $D$ is the diagonal matrix whose entries are $d_{i,i}=   \prod_{k=1}^{i} t_i - t_k$ for $i=1,\ldots,n+1$. 

From Theorem  3 of \cite{MPRNewton},   the upper triangular matrix $U$ in \eqref{eq:U} satisfies  
   \begin{equation}\label{eq:BDUU}
U=  G_1 \cdots G_{n}, 
\end{equation}
where   $G_i\in{\mathbb R}^{(n+1)\times(n+1)}$, $i=1,\ldots,n$, are upper triangular bidiagonal matrices whose structure is described by \eqref{fgi} and their off-diagonal  entries   are  
$$      
   \widetilde{m}_{i,j} = t_{j},    \quad 1\le j< i \le n+1.  
$$
On the other hand, by \eqref{CroutV}, we also have,
 \begin{eqnarray*} 
V = M \left[ \begin{array}{c} w_0,\ldots,w_n \\t_1,\ldots,t_{n+1}\end{array} \right]U=M \left[ \begin{array}{c} w_0,\ldots,w_n \\t_1,\ldots,t_{n+1}\end{array} \right] G_1 G_2\cdots   G_n,
\end{eqnarray*}
and  conclude
 \begin{eqnarray*} 
  &&   L=M \left[ \begin{array}{c} w_0,\ldots,w_n \\t_1,\ldots,t_{n+1}\end{array} \right]  =F_nF_{n-1}\cdots F_1 D.  
\end{eqnarray*}
\end{proof}

Taking into account Theorem \ref{thm:BDAL},  the bidiagonal factorization of the matrix $L$ can be stored by the matrix $BD(L)$ with  
 \begin{equation}\label{eq:BDL}
 BD(L )_{i,j}:=\begin{cases}
 0  , & \textrm{if } i<j, \\[5pt] 
  \prod_{k=1}^{i-1} (t_i - t_k), & \textrm{if } i=j,  \\[5pt]  
     \prod_{k=1}^{j-1}  \frac{t_{i } - t_{i-k }  }{t_{i-1} -t_{i-k-1}  },   & \textrm{if } i>j. 
 \end{cases}
\end{equation}
 
The analysis of the sign of the entries in \eqref{eq:BDL} will allow us to characterize the total positivity property of the collocation matrix of Newton  bases in terms of the ordering of the  nodes. This fact is stated in the following result.

\begin{corollary}  \label{cor:L}
Given interpolation nodes $t_1,\ldots, t_{n+1}$, with $t_i\ne t_j$ for $i\ne j$, let $L\in{\mathbb R}^{(n+1)\times(n+1)}$  be the collocation matrix  \eqref{eq:ColNewton} of the Newton basis  \eqref{eq:NewtonBasis}  and $J$ the diagonal matrix $J:=\textrm{diag}((-1)^{i-1}  )_{1\le i\le n+1}$.   
\begin{enumerate}
\item[a)]  The matrix $L$ is totally positive if and only if  $t_1<\cdots<t_{n+1}$. Moreover, in this case, $L$ and the matrix $BD(L)$ in \eqref{eq:BDL} can be computed to HRA.
\item[b)]  The matrix  
$L_{J} := L J$ is   totally positive   if and only if $t_1>\cdots>t_{n+1}$.   Moreover, in this case, the matrix $BD(L_J)$  can be computed to HRA. 
\end{enumerate}
Furthermore,   the   singular values  and  the inverse matrix  of $L$, as well as  the  solution  of linear systems $ L d = f$, where the entries of $f = (f_1, \ldots , f_{n+1})^T$ have alternating signs, can be performed to HRA.
\end{corollary} 

\begin{proof}  Let  
$L=  F_n \cdots F_{1}D$ be the bidiagonal factorization provided by Theorem \ref{thm:BDAL}.

a) If $t_1<\cdots<t_{n+1}$,  the entries 
$m_{i,j}$ in \eqref{eq:mij}  and $d_{i,i}$ in \eqref{eq:dii}   are all positive and we  conclude that the diagonal matrix $D$  and the bidiagonal matrix factors $F_i$, $i=1,\ldots,n$,   are   totally positive.   Taking into account that the product of totally positive matrices is a totally positive matrix (see Theorem 3.1 of \cite{Ando}), we can guarantee that $L$ is   totally positive. Conversely, if $L$ is totally positive then the entries  $m_{i,j}$ in \eqref{eq:mij}  and $d_{i,i}$ in \eqref{eq:dii}  take all positive values. Moreover, since
  $$
  t_2-t_1= d_{2,2}>0, \quad   t_i-t_{i-1}=m_{i,2}(t_{i-1}-t_{i-2}),  \quad i=3,\ldots,n+1,
 $$
 we derive  by induction  that $t_i-t_{i-1}>0$ for $i=2,\ldots,n+1$. 
 
On the other hand, for increasing sequences of nodes, the subtractions in the
 computation of the entries $m_{i,j}$  and $p_{i,i}$     involve only initial data and so, will not lead to subtractive cancellations. So, the computation to high relative accuracy of the above mentioned algebraic problems can be performed to high relative accuracy using the matrix representation \eqref{eq:BDL} and the Matlab commands in Koev's web page  (see Section 3 of   \cite{Koev1}). 

%For  a decreasing sequence of parameters, 
b) Now, using (\ref{eq:BDLNewton}) and defining  $\widetilde D:=DJ$, we can write
\begin{equation}\label{BDwj}
L_J=  F_n\cdots F_1\widetilde D,
\end{equation}
where   $F_i\in{\mathbb R}^{(n+1)\times(n+1)}$, $i=1,\ldots,n$, are the lower triangular bidiagonal matrices  described in \eqref{fgi}, whose off-diagonal  entries   are  given in     \eqref{eq:mij},  and $\widetilde D\in{\mathbb R}^{(n+1)\times(n+1)}$ is the diagonal matrix $\widetilde D=\textrm{diag}(\widetilde d_{1,1},\ldots, \widetilde d_{n+1})$ with
\begin{equation}       
 \widetilde{d}_{i,i} = (-1)^{i-1}\prod_{k=1}^{i-1} (t_i - t_k),       \quad 1\le   i \le n+1.   \label{eq:diitilde}  
\end{equation}
It is worth noting that,  according to \eqref{BDwj}, the bidiagonal decomposition of   $L_J$   is given by
 \begin{equation}\label{eq:BDLJ}
 BD(L_J )_{i,j}:=\begin{cases}
 0  , & \textrm{if } i<j, \\[5pt] 
  (-1)^{i-1}\prod_{k=1}^{i-1} (t_i - t_k), & \textrm{if } i=j,  \\[5pt]  
     \prod_{k=1}^{j-1}  \frac{t_{i } - t_{i-k }  }{t_{i-1} -t_{i-k-1}  },   & \textrm{if } i>j. 
 \end{cases}
\end{equation}
If  $t_1<\cdots<t_{n+1}$ then
$m_{i,j}>0 $, $ \widetilde{d}_{i,i}>0$ and, using the above reasoning, we conclude that  $L_J$ is totally positive.
Conversely,  if $L_J$ is totally positive and so, $m_{i,j}>0 $, $ \widetilde{d}_{i,i}>0$, taking into account that
  $$
  t_2-t_1=-\widetilde{d}_{2,2}<0, \quad   t_i-t_{i-1}=m_{i,2}(t_{i-1}-t_{i-2}),  \quad i=3,\ldots,n+1,
 $$
 we derive by induction that $t_i-t_{i-1}<0$ for $i=2,\ldots,n+1$. 

For decreasing nodes,  the computation to high relative accuracy of $L_j$, 
  its  singular values,  the inverse matrix $L_{J}^{-1}$ and the resolution of  $L_{J}c= f$, where $f = (f_1, \ldots , f_{n+1})^T$ has alternating signs  can be deduced in a similar way to the increasing case.  Finally, since  $J$ is a unitary matrix, the singular values of  $L_{J} $ coincide with those of  $L$. 
Similarly, taking into account that  
\begin{equation*}
L ^{-1}= JL_{ J}^{-1},
\end{equation*}
we can compute $L^{-1}$ accurately. Finally,  if we have a linear system of equations  $L d= f$, where the elements of  $f = (f_1, \ldots , f_{n+1})^T$  have alternating signs,  we can solve to high relative accuracy the  system  $L_{J}c = f$ and then obtain $d=Jc$.
\end{proof}

Let us observe that the  factorization   \eqref{eq:BDLNewton} of the collocation matrix of the Newton basis corresponding to the nodes $t_1,\ldots,t_{n+1}$ can be used to solve Lagrange  interpolation problems. The Newton form of the Lagrange interpolant can be written as follows

\begin{equation}\label{eq:Laginterpolant}
p_n(t)=\sum_{i=0}^n d_{i+1} w_i(t), 
\end{equation} 
with $d_{i}:=[t_1,\ldots,t_{i}]f$, $i=1,\ldots,n+1$. The computation of the divided differences, which are usually obtained through the recursion  in \eqref{eq:divideddif}, can be alternatively obtained by solving the linear system
$$
L d=f, 
$$ 
with $d:=(d_1,\ldots,d_{n+1})^T$ and $f:=(f_1,\ldots,f_{n+1})^T$, using  the Matlab  function  \verb"TNSolve" in Koev's web page \cite{Koev3} and taking
the matrix form $\eqref{eq:BDA}$ of the bidiagonal decomposition of $L$ as input argument.
Taking into account Corollary \ref{cor:L}, $BD(L)$ (for increasing nodes) and  $BD(L_J)$ (for decreasing nodes)    can be computed to high relative accuracy. Then, if the elements of the vector $f$ are given to high relative accuracy, the computation of the vector $d$ can also be achieved to high relative accuracy.

Finally, the numerical experimentation illustrated in Section \ref{sec:experimentos} will compare the vectors of divided differences obtained using the provided bidiagonal factorization, the recurrence  \eqref{eq:divideddif} and, finally, the Matlab command $\setminus$ for the resolution of linear systems.

 \section{Error analysis and perturbation theory}\label{sec:error}

Let us consider   the collocation matrix    $L\in{\mathbb R}^{(n+1)\times(n+1)}$   of the Newton basis  \eqref{eq:NewtonBasis} at the   nodes  $t_1<\cdots < t_{n+1}$ (see \eqref{eq:ColNewton}).
Now, we  present a procedure for  the efficient  computation of
$ BD(L)$, that is, the matrix representation of the bidiagonal factorization \eqref{BDAfac} of $L$.  

Taking into account  \eqref{eq:BDL}, Algorithms 1 and 2 compute      $m_{i,j}:=BD(L)_{i,j}$, $j<i$, and   $p_{i}:=BD(L)_{i,i}$, respectively. 

\bigskip

  \begin{tabular}{ll}
 
 &    {\bf Algorithm 1: Computation of  }  $m_{i,j}$    \hfill \\[3pt]  
&   {\bf for i :=2  to   n+1} \\[3pt]
& $m_{i,1}:=  1$ \\[3pt]
& \hskip 1cm   \textrm{\bf for j :=2   to i-1}\\[3pt]
& \hskip 1cm    $M:=  \frac{  t_i-t_{i-j+1}  }{ t_{i-1}-t_{i-j} }$ \\[3pt]
& \hskip 1cm   $m_{i,j}:=    m_{i,j-1}  \cdot   M $ \\[3pt]
 & \hskip 1cm  {\bf end j}\\[3pt]
&    {\bf end i }  \\[3pt]
\end{tabular}

\bigskip
 
\begin{tabular}{ll}
 
 &   \textrm{\bf Algorithm 2: Computation of }  $p_{i,i}$   \hfill  \\[3pt]  
& $p_{1}:= 1$\\[3pt] 
&   {\bf for i :=2  to   n+1} \\[3pt] 
& $p_{i}:=1 $\\[3pt] 
& \hskip 1cm {\bf for k :=1   to i-1}\\[3pt] 
& \hskip 1cm  $ p_{i}:= p_{i} \cdot (t_i-t_{i-k}) $ \\[3pt] 
 & \hskip 1cm  {\bf end k} \\[3pt] 
&    {\bf end i }  \\[3pt] 
\end{tabular}

In the sequel, we analyze the stability of Algorithms 1 and 2 under the influence of imprecise computer arithmetic or perturbed input data.
For this purpose, let us first introduce some standard notations in error analysis.

For a given floating-point
arithmetic and a real value $a\in\mathbb R$,  the computed element   is usually  denoted by either $\hbox{fl}(a)$ or by $\hat a$. In order to study the effect of  rounding errors,   we shall use  the well-known models
\begin{equation}\label{(2.1)}
\hbox{fl}(a \, \hbox{op} \,  b)=(a \, \hbox{op} \, b)(1 + \delta)^{\pm 1},
\quad |\delta| \le u,
\end{equation}
where $u$  denotes the unit roundoff and $\hbox{op}$  any of the elementary
operations $+$, $-$, $\times$, $/$ (see \cite{Higham}, p. 40 for more
details). 

 Following \cite{Higham}, when performing an error analysis, one usually deals with quantities  $\theta _k$  such
that  
\begin{equation}\label{eq:gamma}
|\theta _k|\le\gamma _k, \quad  \gamma _k:={ku\over 1-ku},
\end{equation}
for a given $k\in {\mathbb N}$ with $ku<1$.   
Taking into account, Lemmas 3.3 and 3.4 of \cite{Higham}, the
following properties of the values \eqref{eq:gamma} hold:
\begin{itemize}
 \item[a)] $(1+\theta _k)(1+\theta _j)= 1+\theta _{k+j}$,
 \item[b)] $\gamma _k+\gamma _j+\gamma _k\gamma _j\le \gamma
_{k+j}$,
 \item[c)] $\gamma _k+u\le \gamma _{k+1}$,
 \item[d)] if
 $\rho _i=\pm 1$, $\vert \delta _i\vert \le u$, $i=1,\ldots,k$,
then
 $$
 \prod _{i=1}^k(1+\delta _i)^{\rho _i}=1+\theta _k.
 $$
\end{itemize}
For example, statement a) above means that for any given two values $\theta_k$
and $\theta_j$, bounded by $\gamma_k$ and $\gamma_j$, respectively,
 there exists a number $\theta_{k+j}$,   bounded by $\gamma_{k+j}$,
such that the above identity holds.  Further use of the previous
symbols must be intended in this respect.

The following result analyzes the numerical error due to imprecise computer arithmetic in Algorithms 1 and 2, showing that both compute the bidiagonal factorization \eqref{BDAfac} of $L$ accurately in a floating point arithmetic. 

\begin{theorem}  \label{thm:errorBDH}  For $n>1$, let $L\in{\mathbb R}^{(n+1)\times(n+1)}$  be the collocation matrix  \eqref{eq:ColNewton} of the Newton basis  \eqref{eq:NewtonBasis} at the   nodes  $t_1<\cdots < t_{n+1}$.
Let $BD(L)=(b_{i,j})_{1\le i,j\le n+1}$ be the matrix form of the bidiagonal decomposition \eqref{BDAfac} of  $L$ and  
$\hbox{fl}(BD(L ))=(\hbox{fl}( b_{i,j}) )_{1\le i,j\le n+1}$ be the matrix computed with Algorithms 1 and 2 in floating point arithmetic with machine precision $u$. 
Then 
\begin{equation}
\left| \frac{ b_{i,j}- \hbox{fl}( b_{i,j} )}{ b_{i,j}}  \right|    \le  \gamma_{ 4n-5},\quad 1\le i,j\le n+1. 
\end{equation}
\end{theorem}
 
\begin{proof}
For $i> j$,   $b_{i,j}$ can be computed using Algorithm 1. Accumulating  relative errors as proposed in \cite{Higham},  we can easily derive 
 \begin{equation}\label{eq:flmultipliersHq} 
 \left|  \frac{ b_{i,j}- \hbox{fl}( b_{i,j} )}{ b_{i,j}}  \right|  \le   \gamma_{ 4j-5 } \le     \gamma_{ 4 i-9 } \le    \gamma_{ 4n-5 },   \quad  1\le j<i\le n+1.
\end{equation}
Analogously, for $i= j$, $b_{i,i}$ can be computed using Algorithm 2 and    we have
 \begin{equation}\label{eq:flmultipliersHqT} 
\left|  \frac{ b_{i,i}- \hbox{fl}( b_{i,i} )}{ b_{i,i}}  \right| \le   \gamma_{2i-3} \le   \gamma_{2n-1},   \quad  1\le i \le n+1.
\end{equation}
Finally, since $2n-1 \le 4n-5$ for $n>1$, the result follows.
 \end{proof}
 
 Now, we analyze the effect on the bidiagonal factorization of the collocation matrices of Newton  bases due to small relative perturbations in the interpolation nodes $t_i$, $i=1,\ldots,n+1$.  Let us suppose that the perturbed nodes are   
 $$
 t_i'=t_i(1+\delta_i), \quad  i=1,\ldots,n+1.
 $$ 
We 
  define the following values that  will allow us to obtain  an appropriate structured condition number, in a similar way to other analyses performed in   \cite{Demmel2,Koev4,Marco4,Marco2016,Marco5,Marco2022}:
\begin{equation}
rel\_gap_t:={\displaystyle\min_{i\ne j}\frac{ | t_i-t_j | }{|t_i | + | t_j | }}, \quad \theta:=\displaystyle\max_i \frac{ |t_i-t_i' | }{| t_i | }=\displaystyle\max_i \delta_i, \quad \kappa :=\frac{1}{rel\_gap_t},
\end{equation}
 where $rel\_gap_t>>\theta$.

 \begin{theorem} \label{thm:errorBDHpert}
Let $L$ and $L'$  be the collocation matrices  \eqref{eq:ColNewton} of the Newton basis  \eqref{eq:NewtonBasis} at the   nodes  $t_1<\cdots < t_{n+1}$ and $t_1'<\cdots < t_{n+1}'$, respectively, with $t_i'=t_i(1+\delta_i)$, for $i=1,\ldots,n+1$, and $|\delta_i|\le \theta_i$. Let $BD(L)=(b_{i,j})_{1\le i,j\le n+1}$ and  $BD(L')=(b_{i,j}')_{1\le i,j\le n+1}$ be the matrix form of the bidiagonal factorization of $L$ and $L'$, respectively. Then 
\begin{equation}\label{eq:ebperturb}
\left| \frac{ b_{i,j}-  b_{i,j} '}{ b_{i,j}}  \right|    \le  {(2n-2)\kappa \theta\over 1-(2n-2)\kappa  \theta},\quad 1\le i,j\le n+1.. 
\end{equation}
\end{theorem}
 \begin{proof}
First, let us observe that 
$$
t_i'-t_j'= (t_i-t_j) (1+\delta_{i,j}), \quad  |\delta_{i,j}|\le \frac{\theta}{rel\_gap_t}.
$$
Accumulating the perturbations in the style of Higham (see Chapter 3 of \cite{Higham}), we derive
\begin{equation}\label{eq:mijperturbed}
m_{i,j}'=m_{i,j} (1+\bar{\delta}), \quad |\delta|\le  {2(n-1)\kappa \theta\over 1-(2n-2)\kappa  \theta}.
\end{equation}
Analogously, for the diagonal entries $p_{i}$ we have 
\begin{equation}\label{eq:piperturbed}
p_{i}'=p_{i} (1+\bar{\delta}), \quad |\bar{\delta} |\le  { (n-1)\kappa \theta\over 1-(n-1)\kappa  \theta}.
\end{equation}
Finally, using \eqref{eq:mijperturbed} and  \eqref{eq:piperturbed}, the result follows.
\end{proof}

 Formula \eqref{eq:ebperturb} can also be  obtained using  Theorem 7.3 of  \cite{Marco5} where it is shown that small relative perturbations in the nodes of a Cauchy-Vandermonde matrix  produce only small relative perturbations in its bidiagonal factorization. Let us note that the entries $m_{i,j}$ and $p_{i,i}$ of the bidiagonal decomposition of collocation matrices of Newton bases at distinct nodes coincide with those of Cauchy-Vandermonde matrices with $l=0$. Finally, let us note that  quantity $(2n-2)\kappa \theta$ can be seen as an appropriated structured condition number for the mentioned collocation matrices  of the Newton bases  \eqref{eq:NewtonBasis}. 

\section{Numerical experiments}\label{sec:experimentos}

In order to give a numerical support to the theoretical methods discussed in the previous sections,  we provide a series of numerical experiments. In all cases, we have considered  ill-conditioned nonsingular   collocation matrices $L$ of $(n+1)$-dimensional Newton bases \eqref{eq:NewtonBasis} with equidistant increasing or decreasing nodes for $n=15, 25, 50,100$ in a unit-length interval---with the exception of the experiments concerning the Runge function, where the chosen interval is $[-2,2]$.

Let us recall that once the bidiagonal decomposition of a TP matrix $A$ is computed to high relative accuracy, its matrix representation $BD(A)$  can be used as an input argument of the functions of the \verb"TNTool" package, made available by Koev in \cite{Koev3}, to perform to high relative accuracy the solution of different algebraic problems. In particular, \verb"TNSolve" is used to resolve linear systems $Ax=b$ for vectors $b$ with alternating signs, \verb"TNSingularValues" for the computation of the singular values of $A$ and \verb"TNInverseExpand" (see \cite{Marco3}) to obtain the inverse $A^{-1}$. The computational cost  of  \verb"TNSingularValues" is $O(n^3)$, being $O(n^2)$ for the other methods.

In order to check the accuracy of our bidiagonal decomposition approach, $BD(L)$  and $BD(L_{J})$  (see \eqref{eq:BDL} and \eqref{eq:BDLJ})   have been considered with the mentioned routines to solve each problem. The obtained approximations have been compared with those computed by other commonly used procedures, such as standard Matlab routines or the divided difference recurrence \eqref{eq:divideddif}, in the case of the interpolant coefficients.

In this context, the values provided by Wolfram Mathematica $13.1$ with $100$-digit arithmetic have been taken as the exact solution of the considered algebraic problems. Then, the relative errors of the results are given by $e:=|(y-\tilde{y})/y|$, where $\tilde{y}$ is the approximation of the given method and $y$ the exact value computed in Mathematica.

{\bf Computation of coefficients of the Newton form of the Lagrange interpolant to HRA.}
In this numerical experiment, the coefficients $d_{i}:=[t_1,\ldots,t_{i}]f$, $i=1,\ldots,n+1$, of the Newton form of the Lagrange interpolant \eqref{eq:Laginterpolant} were computed using three different methods. First, the divided difference recurrence \eqref{eq:divideddif}  determined $d_i$ directly. Later, the standard $\setminus$ Matlab command and  the bidiagonal decomposition of Section \ref{sec:BDANewton} were also used. Let us note that both of them obtain $d_i$ as the solution of the linear system $Ld=f$, where $L$ is the collocation matrix of the considered Newton basis with $d:=(d_1,\ldots,d_{n+1})^T$, $f:=(f_1,\ldots,f_{n+1})^T$. 
Let us recall that  the last method requires the bidiagonal decomposition  $BD(L)$ (for increasing nodes) and  $BD(L_J)$ (for decreasing nodes)  as an input argument for the Octave/Matlab function \verb"TNSolve", available in \cite{Koev3}. Note that for nodes in decreasing order, the system  $L_Jc=f$,  with $L_J:=L J$ and $J:=\textrm{diag}((-1)^{i-1}  )_{1\le i\le n+1}$,
is solved and  the obtained solution $c$ allows  to recover $d$ as $d=Jc$ (see the proof of Corollary \ref{cor:L}).

In all cases,  $f_i$,  $i=1,\dots,n+1$,  were chosen to be random integers with uniform distribution in $[0,10^3]$, adding alternating signs to guarantee high relative accuracy when  obtaining the coefficients  with the divided difference recurrence \eqref{eq:divideddif} and with \verb"TNSolve". 

Relative errors are shown in Table \ref{TableLinearSystems}. The results clearly illustrate the high relative accuracy achieved  with the divided difference recurrence \eqref{eq:divideddif} and the function \verb"TNSolve" applied to the bidiagonal decompositions proposed in this work, supporting the theoretical results of the previous sections.

Let us note that the extension to the bivariate interpolation with the Lagrange-type data
 using Newton bases and  leading to the use of a generalized Kronecker product of their collocation matrices will be addressed  by the authors by applying
 the results in this paper.

\begin{table}[ht]\centering
\caption{\label{TableLinearSystems} Relative errors of the approximations to the solution of the linear system $Ld=f$, for equidistant nodes in a unit-length interval in increasing and decreasing order.}
\begin{tabular}{@{}ccccccc@{}}
\toprule
%\cmidrule{2-5}
\multicolumn{1}{l}{} &  \multicolumn{3}{c}{\bf $t_1<\cdots<t_{n+1}$ } & \multicolumn{3}{c}{\bf $t_1>\cdots>t_{n+1}$ }  \\
\midrule
$n+1$  & {Div. diff.} &  $ L \setminus f$ & {\verb"TNSolve"$(BD(L),f)$}  &  {Div. diff.} &   $L \setminus f$ & {$J\cdot$\verb"TNSolve"$(BD(L),f)$}    \\
\midrule
 15 & $1.6e-16$&   $1.3e-13$ & $2.4e-17$ &  $8.7e-17$ & $3.0e-17$ & $3.0e-17$\\
25  & $8.7e-16$&   $6.2e-11$ & $8.0e-16$& $6.2e-16$ & $2.2e-10$ & $1.9e-16$ \\
50  &$1.1e-15$  & $1.0e-3$ & $2.1e-15$ & $2.9e-15$ & $2.8e-4$ & $3.8e-15$ \\ 
100 & $4.7e-15$  &$2.6e+12$& $5.8e-15$ &  $5.2e-15$ & $5.8e+12$ & $6.1e-15$\\
\botrule
\end{tabular}
\end{table}

{\bf Computation of the coefficients of the Newton form for a function of constant sign.}
As explained in the previous case, when determining the interpolating polynomial coefficients of the Newton form, to guarantee high relative accuracy the vector $(f_1,\dots,f_{n+1})$ is required to have alternating signs. However, other sign structures can be addressed by the methods discussed in this work, although HRA cannot be assured in these cases. To illustrate the behavior of both the divided differences recurrence and the bidiagonal decomposition approach in such scenario, numerical experiments for the classical Runge function $1/(1+25x^2)$ with equidistant nodes in the $[-2,2]$ interval were performed. Relative errors of the computed coefficients of the Newton form are gathered for several $n$ in Table \ref{TableRunge}, comparing the performance of the divided difference recurrence \eqref{eq:divideddif}, the function \verb"TNSolve" applied to the bidiagonal factorization proposed in Section \ref{sec:BDANewton} and the standard $\setminus$ Matlab command.
As can be seen, good accuracy is achieved for small enough values of $n$, while relative errors for increasing $n$ behave considerably worse with the standard $\setminus$ Matlab command than with both the divided differences recurrence and the bidiagonal decomposition approaches, which achieve similar precision.

\begin{table}[ht]\centering
\caption{\label{TableRunge} Relative errors of the approximations to the solution of the linear system $Ld=f$, for the classical Runge function $(1+25x^2)^{-1}$, with equidistant nodes in increasing order in $[-2,2]$.}
\begin{tabular}{@{}cccc@{}}
\toprule
$n+1$  & {Div. diff.} &  $ L \setminus f$ & {\verb"TNSolve"$(BD(L),f)$}\\
\midrule
 15  & $1.5e-16$&   $9.2e-16$ & $2.5e-16$ \\
 25  & $7.0e-16$&   $5.6e-14$ & $7.2e-16$ \\
 50  & $8.1e-14$&   $2.4e-09$ & $7.5e-14$ \\
 100 & $3.2e-08$&   $7.5e+01$ & $3.3e-08$ \\
\botrule
\end{tabular}
\end{table}

{\bf Computation of singular values of $L$ to HRA.} In the second numerical experiment, the precision achieved by two methods to compute the lowest singular values of  $L$ is analyzed. On the one hand, we computed the lowest singular value of $L$ by  means of the routine \verb"TNSingularValues", using  $BD(L)$ for increasing and $BD(L_{J})$ for decreasing order. Notice that, for nodes in  decreasing order, the singular values of $L_J$  coincide with those of $L$ since $J$ is unitary. On the other hand, the lowest singular value of $L$ was computed with the Matlab command \verb"svd". 

As expected, when the bidiagonal decomposition approach is used, the singular values are computed to high relative accuracy for every dimension $n$, despite the ill-conditioning of the matrices and in contrast with the standard Matlab procedure.
\begin{table}[ht]\centering
\caption{\label{TableVSL}  Relative errors when computing the lowest singular value of $L$ for equidistant nodes in an interval of unit-length in increasing and decreasing order.}
\begin{tabular}{@{}ccccc@{}}
\toprule
%\cmidrule{2-5}
\multicolumn{1}{l}{}  & \multicolumn{2}{c}{ $t_1<\cdots<t_{n+1}$ } & \multicolumn{2}{c}{ $t_1>\cdots>t_{n+1}$ }  \\
\midrule
 $n+1$ & \verb"svd"$(L)$ & {\verb"TNSingV"$(BD(L))$} & \verb"svd"$(L)$ & {\verb"TNSingV"$(BD(L_{J}))$} \\
\midrule
 15 & $7.3e-12$ & $6.6e-16$ & $5.9e-12$ & $5.7e-16$ \\
25 & $1.0e-7$ & $5.2e-16$ & $1.7e-7$ & $4.3e-15$ \\
50 & $1.7e+1$ & $5.1e-16$ & $1.9e+1$ & $8.3e-15$ \\ 
100 & $1.4e+12$ & $1.4e-15$ & $5.1e+13$ & $2.6e-16$ \\
\botrule
\end{tabular}
\end{table}

{\bf Computation of inverse of $L$ to HRA.} Finally, to show another application of the bidiagonal decompositions presented in this work, the inverse of the considered collocation matrices were computed. We  used  two different procedures: the  Matlab \verb"inv" command   and the function \verb"TNInverseExpand"  with our bidiagonal decompositions $BD(L)$ for increasing and $BD(L_{J})$ for decreasing order. To compute the inverse for the decreasing order case to HRA, note that first the inverse of $L_J$ is obtained and then $L^{-1}$ is recovered with $L ^{-1}= JL_{ J}^{-1}$.

It should be pointed out that formula (13) in \cite{CKP} allows to compute $L^{-1}$ to high relative accuracy for any order of the nodes.
Displayed in Table \ref{TableInverses}, the results of these numerical experiments show that our bidiagonal decomposition approach preserves high relative accuracy for any of the values of $n$ tested, in stark contrast with the standard Matlab routine.
\begin{table}[ht]\centering
\caption{\label{TableInverses} Relative errors when computing the inverse of collocation matrices of the Newton basis with equidistant nodes in a unit-length interval in increasing and decreasing order.}
\begin{tabular}{@{}ccccc@{}}
\toprule
%\cmidrule{2-5}
\multicolumn{1}{l}{}  & \multicolumn{2}{c}{ $t_1<\cdots<t_{n+1}$ } & \multicolumn{2}{c}{ $t_1>\cdots>t_{n+1}$ }  \\
\midrule
 $n+1$ & \verb"inv"$(L)$ & {\verb"TNInvEx"$(BD(L))$} & \verb"inv"$(L)$ & {$J\cdot$\verb"TNInvEx"$(BD(L_{J}))$} \\
\midrule
 15 & $2.5e-13$ & $1.3e-15$ & $1.5e-13$ & $5.7e-16$ \\
25 & $8.8e-11$ & $4.8e-15$ & $1.4e-10$ & $8.8e-16$ \\
50 & $2.7e-3$ & $8.7e-15$ & $2.9e-3$ & $2.5e-15$ \\ 
100 & $1.7e+12$ & $6.8e-15$ & $2.8e+12$ & $5.1e-16$ \\
\botrule
\end{tabular}
\end{table}

\backmatter

\section*{Declarations}

\bmhead{Availability of supporting data} The authors confirm that the data supporting the findings of this study are available within the manuscript. The Matlab and Mathematica codes to run the numerical experiments are available upon request.
\bmhead{Competing interests} The authors have no competing interests to declare that are relevant to the content of this article.
\bmhead{Funding} This work was partially supported through the Spanish research grants PID2022-138569NB-I00 and RED2022-134176-T (MCI/AEI) and by Gobierno de Arag\'{o}n (E41$\_$23R, S60$\_$23R).
\bmhead{Authors' contributions} The authors contributed equally to this work.


\end{document}